\documentclass[11pt]{amsart}

\usepackage[T1]{fontenc}
\usepackage{lmodern}
\usepackage{microtype}
\usepackage{mathtools}
\usepackage{amssymb}

\usepackage[margin=1.12in]{geometry}
\usepackage[colorlinks=true,linkcolor=blue,citecolor=black,
            urlcolor=blue]{hyperref}
\hypersetup{
  pdftitle={Nielsen Numbers in Torus-Bundle Skeins},
  pdfauthor={Ahmet Selman Kaya},
  pdfsubject={SLN quantum-torus summands and visible Nielsen numbers},
  pdfkeywords={SLN skein module, quantum-torus summand, empty-skein summand, torus bundle, Nielsen number, finite observer, twisted Hochschild homology}
}

\newtheorem{theorem}{Theorem}[section]
\newtheorem{proposition}[theorem]{Proposition}
\newtheorem{lemma}[theorem]{Lemma}
\newtheorem{corollary}[theorem]{Corollary}
\theoremstyle{definition}
\newtheorem{example}[theorem]{Example}
\theoremstyle{remark}
\newtheorem{remark}[theorem]{Remark}

\newcommand{\Z}{\mathbb Z}
\newcommand{\F}{\mathbb F}
\newcommand{\K}{\mathbb K}
\newcommand{\coker}{\operatorname{coker}}
\newcommand{\im}{\operatorname{im}}
\newcommand{\Sk}{\operatorname{Sk}}
\newcommand{\SkAlg}{\operatorname{SkAlg}}
\newcommand{\HH}{\operatorname{HH}}
\newcommand{\Fix}{\operatorname{Fix}}
\newcommand{\tr}{\operatorname{tr}}
\newcommand{\cyc}{\operatorname{cyc}}
\newcommand{\cont}{\operatorname{cont}}

\title{Nielsen Numbers in Torus-Bundle Skeins}

\author{Ahmet Selman Kaya}


\begin{document}

\begin{abstract}
For a torus bundle with hyperbolic monodromy $\gamma$, we compute the
dimension of the empty-skein summand of its generic $\mathrm{SL}_N$
skein module. Each term in the Weyl centralizer average involves periodic
Nielsen numbers and a count of classes visible modulo a divisor of $N$.
For monodromies of the same trace and any $N\geq2$, the rank-$N$
dimensions agree exactly when the groups of $N$-torsion points fixed
by the two monodromies are isomorphic. The dimensions in all ranks on all
cyclic covers along the base circle determine, and are determined by,
the signed trace and the torsion in the first homology of the double
cover. We also construct infinitely many pairs of Sol torus bundles
whose $\mathrm{GL}_N$ dimensions agree in every rank and whose
finite-cover visibility profiles agree at every iterate, while their
$\mathrm{SL}_3$ empty-skein dimensions differ by six.
\end{abstract}

\maketitle
\enlargethispage{2pt}

\section{Introduction}

Let $f_\gamma$ be the automorphism of $T^2$ induced by a hyperbolic
matrix $\gamma\in\mathrm{SL}_2(\Z)$, and let
$M_\gamma=T^2\times_\gamma S^1$ be its mapping torus.
Bierent--Jordan--Vancraeynest--Vazirani showed that the generic
$\mathrm{GL}_N$ skein dimensions of $M_\gamma$ depend only on
$\tr(\gamma)$ \cite[Theorem~1.2]{BierentJordanVancraeynestVazirani}.
Kinnear's formula for the full $\mathrm{SL}_2$ skein dimension already
involves the invariant factors of $I-\gamma$ and $I+\gamma$
\cite[Theorem~1.2 and Remark~1.7]{Kinnear}.

Write $S_N(\gamma)$ for the dimension of the summand associated with
the empty skein in Kinnear's decomposition,
\[
 S_N(\gamma)=\dim_{\K}
 \HH_0^\gamma\bigl(\SkAlg_{\mathrm{SL}_N}(T^2)\bigr).
\]
Kinnear asked how the Weyl centralizers act on the sectors of this
summand \cite[Question~1.8]{Kinnear}. For hyperbolic monodromy,
Theorem~\ref{thm:general-sln-formula} gives the fixed-point count
for every element $g$ of each centralizer. The free part of the
coinvariant lattice contributes a ratio of periodic Nielsen numbers
$N_k=N(f_\gamma^k)=|\det(I-\gamma^k)|$ \cite{BrooksBrownPakTaylor}. Its torsion
contributes the number of Nielsen classes that remain distinguishable
after reduction modulo $e$, the greatest common divisor of the cycle
lengths of $g$. This is the visible Nielsen number of the cover
associated with $e\Z^2$ \cite{Kaya}. Only $k\leq N$ and $e\mid N$
occur (Corollary~\ref{cor:data-dependence}).

What can one recover from a single dimension? At fixed trace,
$S_N(\gamma)$ and the isomorphism type of the group of fixed points
of $f_\gamma$ on $T^2[N]$ determine one another
(Theorem~\ref{thm:rank-detection}).
The proof regroups the Weyl sum by orders of torsion elements. At a
composite rank, different orders can contribute with opposite signs
when two monodromies are compared. The $\mathrm{GL}_r$ dimensions
grow fast enough that the contribution from the smallest order at which
the counts differ outweighs all contributions of the opposite sign combined.

Passing to the degree-$j$ cover along the base replaces $\gamma$ by
$\gamma^j$. Corollary~\ref{cor:cyclic-reconstruction} recovers the
dimensions on every such cover from the rank-three dimension on the
original bundle and, when needed, one further dimension on its double cover.
In the rank-three examples of Theorem~\ref{thm:infinite-family},
the two groups $\coker(I-\gamma)$ have the same order. One has three
elements killed by $3$ and the other has nine. This accounts for the
difference of six in the empty-skein dimensions.

\section{The quantum-torus summand}

Fix $N\geq2$ and a generic quantum parameter $q$. We work over a
characteristic-zero field $\K$, adjoining the fractional powers of $q$
in Kinnear's normalized monomials when necessary. Scalar extension
preserves the dimensions computed below. Write
\[
  \Lambda_N=\Z^N/\Z(1,\ldots,1),
\]
the character lattice of a maximal torus of $\mathrm{SL}_N$. The Weyl
group $\mathfrak S_N$ permutes the coordinates. The $\mathrm{SL}_N$ skein algebra
of $T^2$ is the $\mathfrak S_N$-invariant subalgebra of the quantum torus on
$\Lambda_N\otimes H_1(T^2;\Z)$
\cite{QuantumCharacter,SkeinsOnTori} (see \cite[Theorem~2.14]{Kinnear}).
For $w\in \mathfrak S_N$, let
$P_w$ be its permutation matrix on $\Z^N$, let $M_w$ be the induced
map on $\Lambda_N$, and put
\[
 C_w^{(N)}(\gamma)
 :=
 \coker\bigl(I-M_w\otimes\gamma:
       \Lambda_N\otimes\Z^2\longrightarrow
       \Lambda_N\otimes\Z^2\bigr).
\]
These groups are finite for hyperbolic $\gamma$. Kinnear's decomposition
\cite[Theorem~1.4 and Corollary~3.12]{Kinnear} gives
\begin{equation}\label{eq:kinnear-decomposition}
 \HH_0^\gamma\bigl(\SkAlg_{\mathrm{SL}_N}(T^2)\bigr)
 \cong
 \bigoplus_{[w]\in\operatorname{cl}(\mathfrak S_N)}
       \K[C_w^{(N)}(\gamma)]_{C(w)},
\end{equation}
where $C(w)$ is the centralizer of $w$ in $\mathfrak S_N$. Before taking
coinvariants, the classes of Kinnear's normalized monomials form a
basis indexed by $C_w^{(N)}(\gamma)$
\cite[Proposition~3.11 and Corollary~3.12]{Kinnear}.
The other summands of the full skein module come from endomorphism
algebras of marked points \cite[Corollary~3.6 and Question~1.8]{Kinnear}.

\begin{lemma}[Phase-freeness of the centralizer action]
\label{lem:phase-free}
Let $\gamma$ be hyperbolic, let $w\in \mathfrak S_N$, and let $g\in C(w)$.  Put
$A_w=M_w\otimes\gamma$, let $\Omega$ denote the invariant skew form on
$\Lambda_N\otimes\Z^2$, and let
\[
 \widetilde m_x
 =q^{\frac12\Omega((I-A_w)^{-1}x,\,x)}\,m_x,
 \qquad x\in\Lambda_N\otimes\Z^2,
\]
denote the normalized monomials \cite[Proposition~3.11]{Kinnear}.
Then
\[
 g(\widetilde m_x)=\widetilde m_{gx}
 \qquad\text{for all }x.
\]
Consequently the $C(w)$-action on the $w$-summand is the permutation
action induced by the natural action of $C(w)$ on
$C_w^{(N)}(\gamma)$, and the dimension of the coinvariants is the
number of $C(w)$-orbits.
\end{lemma}

\begin{proof}
The Weyl action $g(m_x)=m_{gx}$ preserves the multiplication
$m_xm_y=q^{\frac12\Omega(x,y)}m_{x+y}$ since it preserves $\Omega$
\cite[Notation~2.15 and Lemma~2.19]{Kinnear}. The eigenvalues of $A_w$
are of the form $\zeta\lambda^{\pm1}$ with $\zeta$ a root of unity and
$|\lambda|>1$. Thus $I-A_w$ is invertible over $\mathbb Q$ and the
preimage in the normalization is unique \cite[Remark~3.10]{Kinnear}.
The action of $g\in C(w)$ commutes with $A_w$, so
\[
 \Omega\bigl((I-A_w)^{-1}gx,\,gx\bigr)
 =\Omega\bigl(g(I-A_w)^{-1}x,\,gx\bigr)
 =\Omega\bigl((I-A_w)^{-1}x,\,x\bigr),
\]
whence $g(\widetilde m_x)=\widetilde m_{gx}$.

The proof of \cite[Theorem~3.5]{Kinnear} identifies the residual
relations in \eqref{eq:kinnear-decomposition} with $a\sim g(a)$ for
$g\in C(w)$. The coinvariants therefore have one basis vector per orbit.
\end{proof}

\begin{lemma}[Cycle determinant]\label{lem:cycle-determinant}
Let $w\in \mathfrak S_N$ have cycle lengths $c_1,\ldots,c_s$.  If
$\det(I-\gamma^{c_i})\neq0$ for all $i$, then
\[
 \lvert C_w^{(N)}(\gamma)\rvert
 =
 \frac{\prod_{i=1}^s\lvert\det(I-\gamma^{c_i})\rvert}
      {\lvert\det(I-\gamma)\rvert}.
\]
Compare \cite[Lemma~3.1]{BierentJordanVancraeynestVazirani}, where the
corresponding cycle decomposition is obtained on the permutation lattice
used in the $\mathrm{GL}_N$ calculation.
\end{lemma}

\begin{proof}
First use the permutation lattice $\Z^N$.  On the block belonging to a
cycle of length $c$, the eigenvalues of the permutation matrix are the
$c$th roots of unity.  If $\lambda,\lambda^{-1}$ are the eigenvalues of
$\gamma$, then
\[
 \prod_{\zeta^c=1}(1-\zeta\lambda)(1-\zeta\lambda^{-1})
 =
 (1-\lambda^c)(1-\lambda^{-c})
 =
 \det(I-\gamma^c).
\]
Multiplying over the cycles gives
\[
 \det(I-P_w\otimes\gamma\mid\Z^N\otimes\Z^2)
 =
 \prod_i\det(I-\gamma^{c_i}).
\]

Tensor the diagonal sequence
\[
 0\longrightarrow\Z(1,\ldots,1)\longrightarrow\Z^N
 \longrightarrow\Lambda_N\longrightarrow0
\]
with $\Z^2$. On the diagonal, $P_w\otimes\gamma$ restricts to
$\gamma$. The determinant on the quotient is therefore the product
above divided by $\det(I-\gamma)$, which is nonzero under the stated
hypothesis. Taking absolute values gives the cokernel order.
\end{proof}

\begin{lemma}[Fixed points on a difference cokernel]
\label{lem:fixed-cokernel}
Let $L$ be a lattice, let $F:L\to L$ be injective with finite cokernel,
and let $g\in\operatorname{Aut}(L)$ commute with $F$.  Then
\[
 \bigl|\Fix(g\curvearrowright\coker F)\bigr|
 =
 [L:FL+(g-I)L].
\]
Equivalently, the right-hand side is the order of the cokernel of the
endomorphism induced by $F$ on $L/(g-I)L$.
For the corresponding cyclic-sector calculation, compare
\cite[Lemma~3.8]{BierentJordanVancraeynestVazirani}.
\end{lemma}

\begin{proof}
Put $C=\coker F$.  The endomorphism $g-I$ of the finite abelian group
$C$ has kernel and cokernel of the same order.  Its cokernel is
\[
 C/(g-I)C
 \cong L/(FL+(g-I)L),
\]
which proves the formula.  Since $F$ commutes with $g$, it descends to
$L/(g-I)L$, giving the equivalent description.
\end{proof}

\section{The dimension formula}

For $k,m\geq1$, write
\[
 N_k=\lvert\det(I-\gamma^k)\rvert,
 \qquad
 V_m=V_{m\Z^2}(f_\gamma).
\]
The subgroup $m\Z^2$ defines a cover of degree $m^2$. Its visible
Nielsen number $V_m$ counts the classes in the support of the
Reidemeister trace after reduction modulo $m$
\cite[Definitions~2.2--2.3]{Kaya}. If $\gamma_m$
denotes reduction modulo $m$, then
\begin{equation}\label{eq:general-observer-identity}
 V_m
 =\left|\coker\bigl(I-\gamma_m:(\Z/m\Z)^2\to(\Z/m\Z)^2\bigr)\right|
 =\left|\ker\bigl(I-\gamma_m:(\Z/m\Z)^2\to(\Z/m\Z)^2\bigr)\right|.
\end{equation}
Indeed, the fixed-point classes are indexed by $\coker(I-\gamma)$,
which maps onto $\coker(I-\gamma_m)$ under reduction. All classes
have index $\operatorname{sgn}\det(I-\gamma)$, so every image class
has a nonzero coefficient. The second equality holds for every
endomorphism of a finite abelian group.

Let $g\in \mathfrak S_N$ have cycles $O_1,\ldots,O_s$ of lengths
$c_1,\ldots,c_s$, and put
\[
 Q_g:=\Lambda_N/(g-I)\Lambda_N,
 \qquad
 e(g):=\gcd(c_1,\ldots,c_s).
\]

\begin{lemma}[The Weyl coinvariant lattice]
\label{lem:weyl-coinvariant-lattice}
There is an exact sequence
\begin{equation}\label{eq:coinvariant-exact-sequence}
 0\longrightarrow\Z
 \xrightarrow{\;1\mapsto(c_1,\ldots,c_s)\;}
 \Z^s\longrightarrow Q_g\longrightarrow0.
\end{equation}
Consequently,
\[
 Q_g\cong\Z^{s-1}\oplus\Z/e(g)\Z.
\]
If $w\in C(g)$, then $w$ permutes the cycles of $g$ and
\eqref{eq:coinvariant-exact-sequence} is equivariant for the induced
permutation $\bar w$ of $\{O_1,\ldots,O_s\}$, with trivial action on the
left-hand copy of $\Z$.
\end{lemma}

\begin{proof}
The coinvariants of the permutation lattice $\Z^N$ under $g$ are
freely generated by the cycles $O_1,\ldots,O_s$.  In these coordinates
the diagonal vector $(1,\ldots,1)$ maps to
$(c_1,\ldots,c_s)$, which gives
\eqref{eq:coinvariant-exact-sequence}.  Smith normal form gives the
displayed decomposition of $Q_g$.

An element of $C(g)$ permutes cycles of equal length. Rotations within
a cycle act trivially on its coinvariant generator, so $\bar w$ fixes
$(c_1,\ldots,c_s)$ and the source of the first map.
\end{proof}

\begin{lemma}[Triviality on coinvariant torsion]
\label{lem:trivial-torsion-action}
For every $w\in C(g)$, the induced action of $\bar w$ on
$\operatorname{Tor}(Q_g)\cong\Z/e(g)\Z$ is trivial.
\end{lemma}

\begin{proof}
Inside $\Z^s$, the saturation of
$\Z(c_1,\ldots,c_s)$ is generated by
$(c_1/e(g),\ldots,c_s/e(g))$.  Hence
\[
 \operatorname{Tor}(Q_g)
 \cong
 \frac{\Z(c_1/e(g),\ldots,c_s/e(g))}
      {\Z(c_1,\ldots,c_s)}.
\]
The permutation $\bar w$ fixes the displayed generator because it
permutes only cycles of equal length.  It therefore acts trivially on
the quotient.
\end{proof}

For commuting $w,g\in \mathfrak S_N$, let
\[
 d(w,g)=(d_1,\ldots,d_\ell)
\]
be the cycle lengths of the permutation induced by $w$ on the set of
cycles of $g$.

\begin{theorem}[General $\mathrm{SL}_N$ empty-skein formula]
\label{thm:general-sln-formula}
Let $N\geq2$ and let $\gamma\in\mathrm{SL}_2(\Z)$ be hyperbolic.  For
every $w\in \mathfrak S_N$ and $g\in C(w)$,
\begin{equation}\label{eq:general-fixed-term}
 \left|\Fix\bigl(g\curvearrowright C_w^{(N)}(\gamma)\bigr)\right|
 =
 \frac{\prod_{j=1}^{\ell}N_{d_j}}{N_1}\,
 V_{e(g)}.
\end{equation}
Consequently,
\begin{equation}\label{eq:general-sln-dimension}
 \dim\HH_0^\gamma\bigl(\SkAlg_{\mathrm{SL}_N}(T^2)\bigr)
 =
 \sum_{[w]\in\operatorname{cl}(\mathfrak S_N)}
 \frac{1}{|C(w)|}
 \sum_{g\in C(w)}
 \frac{\prod_{j=1}^{\ell(w,g)}N_{d_j(w,g)}}{N_1}\,
 V_{e(g)}.
\end{equation}
\end{theorem}

\begin{proof}
Fix $w\in \mathfrak S_N$ and $g\in C(w)$.  Apply
Lemma~\ref{lem:fixed-cokernel} with
\[
 L=\Lambda_N\otimes\Z^2,
 \qquad
 F=I-M_w\otimes\gamma.
\]
The fixed-point count is the cokernel order of the induced map on
$Q_g\otimes\Z^2$, namely
\[
 F_{w,g}=I-\bar w\otimes\gamma.
\]

Tensoring \eqref{eq:coinvariant-exact-sequence} with the free lattice
$\Z^2$ preserves exactness. The difference
maps on its three terms are $I-\gamma$, $I-P_{\bar w}\otimes\gamma$
and $F_{w,g}$. The first two are injective because the eigenvalues of
$P_{\bar w}$ have modulus one and those of $\gamma$ do not.
The snake lemma gives
\[
 0\longrightarrow\ker F_{w,g}
 \longrightarrow\coker(I-\gamma)
 \longrightarrow\coker(I-P_{\bar w}\otimes\gamma)
 \longrightarrow\coker F_{w,g}
 \longrightarrow0.
\]
By the cycle determinant identity,
\[
 \left|\coker(I-P_{\bar w}\otimes\gamma)\right|
 =\prod_{j=1}^{\ell}N_{d_j},
 \qquad
 \left|\coker(I-\gamma)\right|=N_1.
\]
It follows that
\begin{equation}\label{eq:snake-cardinality}
 |\coker F_{w,g}|
 =|\ker F_{w,g}|\,
   \frac{\prod_{j=1}^{\ell}N_{d_j}}{N_1}.
\end{equation}

The same exact sequence shows that $\ker F_{w,g}$ is finite, hence it
lies in the torsion subgroup of $Q_g\otimes\Z^2$.  By
Lemmas~\ref{lem:weyl-coinvariant-lattice} and
\ref{lem:trivial-torsion-action}, this subgroup is
$(\Z/e(g)\Z)^2$ and $\bar w$ acts trivially on it.  Therefore
\[
 |\ker F_{w,g}|
 =\left|\ker\bigl(I-\gamma_{e(g)}:
       (\Z/e(g)\Z)^2\to(\Z/e(g)\Z)^2\bigr)\right|
 =V_{e(g)}
\]
by \eqref{eq:general-observer-identity}.  Substitution in
\eqref{eq:snake-cardinality} proves
\eqref{eq:general-fixed-term}.  Finally, Burnside's lemma applied to
each summand of \eqref{eq:kinnear-decomposition} gives
\eqref{eq:general-sln-dimension}.
\end{proof}

In rank three, Proposition~\ref{prop:three-sectors} computes the three
sector contributions separately.

\begin{corollary}[Data dependence at rank $N$]
\label{cor:data-dependence}
For hyperbolic $\gamma$, the dimension of the rank-$N$ empty-skein summand is determined
by the periodic Nielsen numbers $N_1,\ldots,N_N$ together with the
visible Nielsen numbers $V_{e\Z^2}(f_\gamma)$ for $e\mid N$.
\end{corollary}

\begin{proof}
Every $d_j(w,g)$ is at most $N$.  Moreover, $e(g)$ divides every cycle
length of $g$, hence it divides their sum $N$.  These are precisely the
data appearing in \eqref{eq:general-sln-dimension}.
\end{proof}

\begin{proposition}[No torsion corrections for $\mathrm{GL}_N$]
\label{prop:gln-freeness}
Let $\gamma\in\mathrm{SL}_2(\Z)$ be hyperbolic and let $N\geq1$.  For
every $w\in \mathfrak S_N$ and $g\in C(w)$,
\[
 \left|\Fix\Bigl(g\curvearrowright
 \coker\bigl(I-P_w\otimes\gamma:
 \Z^N\otimes\Z^2\to\Z^N\otimes\Z^2\bigr)\Bigr)\right|
 =\prod_{j=1}^{\ell(w,g)}N_{d_j(w,g)},
\]
and consequently
\[
 \dim\Sk_{\mathrm{GL}_N}(M_\gamma)
 =\sum_{[w]\in\operatorname{cl}(\mathfrak S_N)}\frac1{|C(w)|}
  \sum_{g\in C(w)}\ \prod_{j=1}^{\ell(w,g)}N_{d_j(w,g)}.
\]
In particular the $\mathrm{GL}_N$-skein dimension is determined by
$N_1,\ldots,N_N$ alone, hence by $\tr(\gamma)$
(Remark~\ref{rem:nielsen-trace}), and no visible-Nielsen correction
occurs: the coinvariant lattice $\Z^N/(g-I)\Z^N$ is free, so the
torsion mechanism of Lemmas~\ref{lem:weyl-coinvariant-lattice}
and~\ref{lem:trivial-torsion-action} is absent.
\end{proposition}

\begin{proof}
For $\mathrm{GL}_N$, Kinnear's decomposition on the permutation
lattice $\Z^N$ gives the whole module \cite[Corollary~3.6]{Kinnear},
\[
 \Sk_{\mathrm{GL}_N}(M_\gamma)
 \cong
 \bigoplus_{[w]\in\operatorname{cl}(\mathfrak S_N)}
 \K\bigl[\coker(I-P_w\otimes\gamma)\bigr]_{C(w)}.
\]
Lemma~\ref{lem:phase-free} applies to the permutation lattice, so
Burnside's lemma gives the second formula from the first. For the
first, apply Lemma~\ref{lem:fixed-cokernel} to $F=I-P_w\otimes\gamma$.
The $g$-coinvariants of $\Z^N$ have one free generator per cycle of
$g$ (Lemma~\ref{lem:weyl-coinvariant-lattice}). On their tensor
product with $\Z^2$, $F$ induces $I-P_{\bar w}\otimes\gamma$.
The free-lattice determinant calculation in
Lemma~\ref{lem:cycle-determinant} gives its cokernel order as
$\prod_j\lvert\det(I-\gamma^{d_j})\rvert=\prod_j N_{d_j}$.
Every $d_j$ is at most $N$.
\end{proof}

\begin{remark}[The $\mathrm{GL}_N/\mathrm{SL}_N$ contrast]
\label{rem:gln-sln-contrast}
Proposition~\ref{prop:gln-freeness} recovers the trace dependence in
\cite{BierentJordanVancraeynestVazirani}. Passing from $\Z^N$ to
$\Lambda_N$ lowers the free rank of the $g$-coinvariants by one and
introduces torsion $\Z/e(g)\Z$
(Lemma~\ref{lem:weyl-coinvariant-lattice}). In
\eqref{eq:general-fixed-term} these changes account for the divisor
$N_1$ and the factor $V_{e(g)}$, respectively.
\end{remark}

\begin{remark}[Diagonal and off-diagonal terms]
If $g=w$ has $s$ cycles, then $w$ induces the identity permutation on
its own cycles.  Thus $d(w,w)=(1,\ldots,1)$ and
\eqref{eq:general-fixed-term} reduces to
\[
 \left|\Fix\bigl(w\curvearrowright C_w^{(N)}(\gamma)\bigr)\right|
 =N_1^{s-1}V_{e(w)}.
\]
For $N=3$, put $\sigma=(123)$. The diagonal pair
$(w,g)=(\sigma,\sigma)$ contributes $V_3/3$ after Burnside averaging.
The pair $(\sigma,\sigma^2)$ contributes another $V_3/3$, and
$(e,\sigma)$ and $(e,\sigma^2)$ together contribute the remaining
$V_3/3$. Only one third of the correction is diagonal.
\end{remark}

\begin{example}[The rank-four formula]\label{ex:sl4-formula}
At $N=4$, Theorem~\ref{thm:general-sln-formula} gives the following
centralizer averages, with $V_e=V_{e\Z^2}(f_\gamma)$.

\begin{center}
\small
\renewcommand{\arraystretch}{1.35}
\begin{tabular}{@{}cl@{}}
\hline
cycle type of $w$ & contribution of the $[w]$-sector \\
\hline
$1^4$
& $\displaystyle \frac{N_1^3}{24}+\frac{N_1^2}{4}
   +\frac{N_1V_2}{8}+\frac{N_1}{3}+\frac{V_4}{4}$ \\
$2\,1^2$
& $\displaystyle \frac{N_1^2}{4}+\frac{N_1N_2}{4}
   +\frac{N_1V_2}{4}+\frac{N_2}{4}$ \\
$2^2$
& $\displaystyle \frac{N_2^2}{8N_1}+\frac{N_2}{4}
   +\frac{N_1V_2}{8}+\frac{N_2V_2}{4N_1}+\frac{V_4}{4}$ \\
$3\,1$
& $\displaystyle \frac{2N_1}{3}+\frac{N_3}{3}$ \\
$4$
& $\displaystyle \frac{N_4}{4N_1}
   +\frac{N_2V_2}{4N_1}+\frac{V_4}{2}$ \\
\hline
\end{tabular}
\end{center}

For example, if $w$ is a $4$-cycle, then
$C(w)=\{e,w,w^2,w^3\}$.  The four fixed-point terms in that sector are
\[
 \frac{N_4}{N_1},\qquad V_4,\qquad
 \frac{N_2}{N_1}V_2,\qquad V_4,
\]
whose average gives the last row. Summing the five sectors gives
\begin{align*}
 \dim\HH_0^\gamma\bigl(\SkAlg_{\mathrm{SL}_4}(T^2)\bigr)
 ={}&\frac{N_1^3}{24}+\frac{N_1^2}{2}
 +\frac{N_1N_2}{4}+\frac{N_1V_2}{2}
 +N_1+\frac{N_2}{2}+\frac{N_3}{3}+V_4\\
 &+\frac{N_2^2}{8N_1}+\frac{N_2V_2}{2N_1}
 +\frac{N_4}{4N_1}.
\end{align*}
The two nontrivial moduli are $2$ and $4$, as in
Corollary~\ref{cor:data-dependence}.
For the companion matrices
$\gamma_t=\left(\begin{smallmatrix}0&-1\\1&t\end{smallmatrix}\right)$,
the values at $t=3,-3,5,-5,6$ are, respectively,
$29,45,139,163,248$.
\end{example}

\section{Torsion fixed points and cyclic covers}
\label{sec:rank-detection}

Let $T^2[N]$ be the subgroup of points killed by $N$ in $T^2$.
For $\gamma\in\mathrm{SL}_2(\Z)$ write $A_\gamma=\coker(I-\gamma)$.
When $\gamma$ is hyperbolic with trace $t$, this group has order $a=|2-t|=N_1$.
Write $s_\gamma$ for the positive gcd of the entries of $I-\gamma$.
The Smith decomposition is
\[
 A_\gamma\simeq \Z/s_\gamma\Z\oplus\Z/(a/s_\gamma)\Z .
\]
Thus $s_\gamma$ divides $E(a)$, where $E(b)$ denotes the largest
positive integer whose square divides $b$.

\begin{theorem}[Detection at a single rank]\label{thm:rank-detection}
Let $N\ge2$. If $\gamma,\gamma'\in\mathrm{SL}_2(\Z)$ are hyperbolic
and have the same trace $t$, then
\[
 S_N(\gamma)=S_N(\gamma')
 \quad\Longleftrightarrow\quad
 \Fix\bigl(f_\gamma|_{T^2[N]}\bigr)
 \cong\Fix\bigl(f_{\gamma'}|_{T^2[N]}\bigr).
\]
Here $\cong$ denotes isomorphism of abelian groups.
Equivalently, the dimensions agree precisely when
$\gcd(N,s_\gamma)=\gcd(N,s_{\gamma'})$.
At fixed $t$, rank $N$ determines $A_\gamma$ for all such monodromies
if and only if $E(|2-t|)\mid N$.
\end{theorem}

\begin{corollary}[Cyclic covers along the base]\label{cor:cyclic-reconstruction}
For hyperbolic $\gamma\in\mathrm{SL}_2(\Z)$, the family
$\bigl(S_N(\gamma^j)\bigr)_{N\ge2,\,j\ge1}$ and the pair
\[
 \bigl(\tr(\gamma),\,\operatorname{Tor} H_1(M_{\gamma^2};\Z)\bigr)
\]
determine one another. The rank-three dimension on $M_\gamma$ gives
the signed trace $t$. If $E(t^2-4)>1$, one further dimension, at rank
$E(t^2-4)$ on the double cover $M_{\gamma^2}$, suffices to recover
the family. If $E(t^2-4)=1$, the trace alone suffices. No finite set
of ranks chosen independently of $t$ determines the family for all
hyperbolic monodromies, even if those ranks are known at every cover
degree.
\end{corollary}

For the action on finite torus lattices, Baake--Roberts--Weiss
give a criterion for two toral endomorphisms to be conjugate on
$T^2[n]$ for every $n$ \cite[Theorem~2 and Corollary~3]{BaakeRobertsWeiss}.

\begin{proof}[Proof of Theorem~\ref{thm:rank-detection}]
Fix $t$, and let $P_r=\dim\Sk_{\mathrm{GL}_r}(M_\gamma)$ for $r\ge1$,
with $P_0=1$. These numbers depend only on $t$ by
Proposition~\ref{prop:gln-freeness}. Write $J_A(d)$ for the number of
elements of exact order $d$ in $A=A_\gamma$. The map
$[v]\mapsto (I-\gamma)^{-1}v\bmod\Z^2$ identifies $A$ with
$\Fix(f_\gamma)$. By \eqref{eq:general-observer-identity},
$V_e=|A[e]|=\sum_{d\mid e}J_A(d)$.
We claim that \eqref{eq:general-sln-dimension} becomes
\begin{equation}\label{eq:rank-divisor}
 aS_N(\gamma)=\sum_{d\mid N}J_A(d)P_{N/d}.
\end{equation}

To see this, write the centralizer average as $1/N!$ times the sum
over commuting pairs $(w,g)$ and group the pairs by the cycle type of $g$.
A term with
$J_A(d)$ occurs exactly when every cycle length of $g$ is divisible
by $d$. Suppose $g$ has $m_c$ cycles of length $c$. Each permutation
of these cycles has $c^{m_c}$ lifts in the corresponding centralizer
factor, one for each choice of rotations. This cancels $c^{m_c}$ in
the order of that factor. The
contribution from this length is
\[
 h_{m_c},\qquad
 h_m=\frac1{m!}\sum_{\sigma\in\mathfrak S_m}
       \prod_{C\in\cyc(\sigma)}N_{|C|},\qquad h_0=1 ,
\]
where the product ranges over the cycles of $\sigma$.
Putting $H(z)=\sum_{m\ge0}h_m z^m$, the unrestricted average gives
$\sum_{r\ge0}P_r z^r=\prod_{c\ge1}H(z^c)$ by
Proposition~\ref{prop:gln-freeness}, in agreement with the partition
function of \cite{BierentJordanVancraeynestVazirani}. Restricting $c$
to multiples of $d$ replaces $z$ by $z^d$, proving
\eqref{eq:rank-divisor}.

Let $T=|t|$, $\epsilon=\operatorname{sgn}(t)$ and
$\rho=(T+\sqrt{T^2-4})/2$. Then $N_j=\rho^j+\rho^{-j}-2\epsilon^j$.
The cycle index of the symmetric groups gives
\[
 H(z)=\exp\left(\sum_{j\ge1}N_j\frac{z^j}{j}\right)
     =\frac{(1-\epsilon z)^2}{1-Tz+z^2}
     =1+\frac{az}{1-Tz+z^2}.
\]
For $r\ge1$ this yields
$h_r=a(\rho^r-\rho^{-r})/(\rho-\rho^{-1})$. In particular
$h_{r+1}/h_r>\rho>a/2$, since $\sqrt{T^2-4}>2$ and $a\le T+2$.
At $r=0$ we have $h_0=1$ and $h_1=a$, so $h_1\ge h_0$ and
$2h_1>ah_0$. The series $\prod_{c\ge2}H(z^c)$ has nonnegative
coefficients and constant term $1$. Convolving its coefficients
with those of $H$ gives
\begin{equation}\label{eq:rank-growth}
 P_{r+1}\ge P_r,\qquad 2P_{r+1}>aP_r\qquad (r\ge0).
\end{equation}

For a prime $p$, set $v=v_p(a)$ and $b=v_p(s_\gamma)\le v/2$. The
Smith decomposition gives
\[
 |A[p^j]|=p^{\min(j,b)+\min(j,v-b)} .
\]
For $j\le v_p(N)$ this depends only on $\min(b,v_p(N))$ and $v$.
Thus equal values of $\gcd(N,s_\gamma)$ give equal $|A[p^j]|$ for
all $p^j\mid N$. Successive differences give $J_A(p^j)$, and
multiplicativity gives $J_A(d)$ for every $d\mid N$.
Equation~\eqref{eq:rank-divisor} proves equality of the dimensions.

For the converse, suppose the gcds differ and put $A'=A_{\gamma'}$.
At a prime where the truncated exponents differ, denote the smaller
and larger exponents of the small Smith factors by $b<b'$. The
first differing order at this prime is $p^{b+1}\mid N$. Indeed
$v\ge 2b+2$, so the groups have $p^{2b+1}$ and $p^{2b+2}$ elements
killed by $p^{b+1}$, respectively, and the same number killed by $p^b$.
Thus the exact-order counts at $p^{b+1}$ differ by
$p^{2b+1}(p-1)\ge2$. By multiplicativity, the least divisor $d_0$ of
$N$ with $J_A(d_0)\ne J_{A'}(d_0)$ is one of these prime powers.

Interchange $A,A'$ if necessary so that the difference at $d_0$ is
positive, and put $k=N/d_0$. Its contribution to the difference of
the sums in \eqref{eq:rank-divisor} is at least $2P_k$. For $d>d_0$
one has $N/d\le k-1$ whenever $d\mid N$. The sum of the magnitudes of
the negative contributions is at most
\[
 P_{k-1}\sum_{\substack{d\mid N\\J_A(d)<J_{A'}(d)}}J_{A'}(d)
 \ \le\ aP_{k-1}\ <\ 2P_k
\]
by \eqref{eq:rank-growth}. The dimensions therefore differ. A related
first-difference argument for zero-sum multisets uses binomial weights
\cite[Theorem~3]{LiZhang}. For equal group orders and a common
degree at least the group order, see
\cite[Lemma~4.4 and Appendix~A]{HanWangZhangZhang}.

It remains to interpret the gcd. The group $\Fix(f_\gamma|_{T^2[N]})$
identifies with $A[N]$, whose invariant factors, in divisibility
order, are $\gcd(N,s_\gamma)$ and $\gcd(N,a/s_\gamma)$.
At fixed $a$ the first determines the second. At any $p\mid N$ it
either specifies $b$ exactly or gives $b\ge v_p(N)$, in which case
$v-b\ge b\ge v_p(N)$ as well. Consequently two such groups are
isomorphic precisely when their first factors agree.

Finally, every $s$ with $s^2\mid a$ occurs as $s_\gamma$ for
\begin{equation}\label{eq:smith-realization}
 \gamma_{t,s}=\begin{pmatrix}1&(t-2)/s\\ s&t-1\end{pmatrix}.
\end{equation}
If $E(a)\mid N$, then $\gcd(N,s_\gamma)=s_\gamma$. Otherwise there
is a prime $p$ with $u=v_p(N)<v_p(E(a))$. The choices $s=p^u$ and
$s=p^{u+1}$ give the same gcd and distinct Smith factors.
\end{proof}

\begin{proof}[Proof of Corollary~\ref{cor:cyclic-reconstruction}]
First we recover the signed trace. In the notation of the preceding
proof, $P_1=a$ and $P_3=h_3+h_1^2+h_1=a(T^2+a)$. In rank three,
\eqref{eq:rank-divisor} reads
\[
 S_3(\gamma)=T^2+a+|A_\gamma[3]|-1 .
\]
The last group has order $1,3$ or $9$. Hence the offset
$S_3(\gamma)-(T^2+T)$ lies in $\{-2,0,6\}$ for $t>0$ and in
$\{2,4,10\}$ for $t<0$. For $T\ge6$ the intervals
$[T^2+T-2,T^2+T+10]$ for consecutive $T$ are disjoint.
For $T=3,4,5$, substituting the possible groups of order $|2-t|$
gives $\{10,14\}$, $\{18,24\}$ and $\{30,32\}$, respectively.
These sets are disjoint and lie below $40$, the lower endpoint for
$T=6$. Thus the dimension specifies both $T$ and the sign of $t$.

Write $\cont(B)$ for the positive gcd of the entries of a nonzero
integral matrix $B$, and set $s_j=\cont(\gamma^j-I)$. The homology
sequence of the mapping torus identifies
$\operatorname{Tor}H_1(M_{\gamma^j};\Z)$ with $A_{\gamma^j}$
\cite[Corollary~3]{MartinsRodriguesSousaRamos}.
Its Smith factors are $s_j$ and $N_j/s_j$. We shall recover them all
from $t$ and $s_2$. First,
\begin{equation}\label{eq:cover-content}
 s_1=\frac{\gcd(s_2,t-2)}{\gcd(2,t)} .
\end{equation}
To prove this, write $\gamma=I+s_1B$ with $\cont(B)=1$.
Then $t-2=s_1\tr B$ and $\tr B=-s_1\det B$ by $\det\gamma=1$.
Right multiplication by $\gamma^{-1}$ preserves content, so
$s_2=\cont(\gamma-\gamma^{-1})=s_1\cont(2B-(\tr B)I)$.
Taking the gcd with $t-2$ gives $s_1\gcd(2,\tr B)$ because $B$ is
primitive. When $s_1$ is odd, $t$ and $\tr B$ have the same parity.
When $s_1$ is even, both are even. This proves \eqref{eq:cover-content}.

Let $u_0=0$, $u_1=1$ and $u_{k+1}=tu_k-u_{k-1}$.
Cayley--Hamilton gives $\gamma^k-\gamma^{-k}=u_k(\gamma-\gamma^{-1})$
and $\gamma^i+\gamma^{-i}=\tr(\gamma^i)I$
\cite[Lemma~1]{BaakeRobertsWeiss}. Hence, for $k\ge1$,
\begin{align*}
 \gamma^{2k}-I&=u_k\,\gamma^{k-1}(\gamma^2-I),\\
 \gamma^{2k+1}-I&=
 \left(1+\sum_{i=1}^{k}\tr(\gamma^i)\right)\gamma^k(\gamma-I).
\end{align*}
The second identity follows by pairing the terms of
$\sum_{r=-k}^{k}\gamma^r$. Taking contents gives every $s_j$, since
the scalar factors depend only on $t$ and the powers of $\gamma$
are unimodular. Thus $t$ and $A_{\gamma^2}$ determine every
$A_{\gamma^j}$ and $\tr(\gamma^j)$, hence every dimension by
Theorem~\ref{thm:rank-detection}.

Conversely, once $S_3(\gamma)$ gives $t$, the double cover has trace
$t^2-2$ and torsion of order $t^2-4$. If $E(t^2-4)=1$ this group
is cyclic. Otherwise Theorem~\ref{thm:rank-detection} recovers it at
rank $E(t^2-4)$. This proves the reconstruction in both directions.
For some traces the second measurement is redundant even when
$E(t^2-4)>1$. For instance, at $t=4$ one has
$s_2=2\cont(\gamma-2I)=2$, since $\cont(\gamma-2I)^2$ divides
$|\det(\gamma-2I)|=3$.

To see why no fixed finite set of ranks suffices, choose a prime $p$
dividing none of them and take $t=2+p^2$. The matrices $\gamma_{t,1}$
and $\gamma_{t,p}$ of \eqref{eq:smith-realization} are conjugate over
$\Z[1/p]$ by $\operatorname{diag}(p,1)$. This diagonal matrix is
invertible modulo every chosen rank, so at each such rank all powers
of the two matrices have isomorphic fixed subgroups.
Theorem~\ref{thm:rank-detection} gives equal dimensions there. At
rank $p$ the two gcds on the base are $1$ and $p$.
\end{proof}

\section{Rank three and Sol torus bundles}

\subsection{The three Weyl sectors}

Write $\Lambda=\Lambda_3$ and
$C_w(\gamma)=C_w^{(3)}(\gamma)$.  The three conjugacy classes of $\mathfrak S_3$
are represented by
\[
 e,\qquad \tau=(12),\qquad \sigma=(123).
\]
In the basis $\bar e_1,\bar e_2$ of $\Lambda$, one may take
\[
 M_\tau=
 \begin{pmatrix}0&1\\1&0\end{pmatrix},
 \qquad
 M_\sigma=
 \begin{pmatrix}0&-1\\1&-1\end{pmatrix}.
\]
The centralizer orders are $6$, $2$, and $3$, respectively.

Write $T_3(\gamma)=|A_\gamma[3]|$ for the number of elements of
$A_\gamma$ killed by $3$.

\begin{proposition}\label{prop:three-sectors}
The orbit counts for the identity, transposition, and $3$-cycle sectors
are, respectively,
\begin{align}
 d_e(\gamma)
 &=\frac{N_1^2+3N_1+2T_3(\gamma)}{6},
 \label{eq:identity-sector}\\
 d_\tau(\gamma)
 &=\frac{N_2+N_1}{2},
 \label{eq:transposition-sector}\\
 d_\sigma(\gamma)
 &=\frac{N_3/N_1+2T_3(\gamma)}{3}.
 \label{eq:threecycle-sector}
\end{align}
\end{proposition}

\begin{proof}
We use Burnside's lemma with Lemma~\ref{lem:fixed-cokernel}.

\smallskip
\noindent\emph{Identity sector.}
Here $C(e)=\mathfrak S_3$.  By Lemma~\ref{lem:cycle-determinant}, the identity
element fixes all $N_1^2$ elements of $C_e(\gamma)$.  For a transposition
$g$, the Smith normal form of $M_g-I$ on $\Lambda$ is
$\operatorname{diag}(1,0)$, so
\[
 \Lambda/(g-I)\Lambda\cong\Z.
\]
The map $I-I_\Lambda\otimes\gamma$ induced on the quotient is
$I-\gamma$, and Lemma~\ref{lem:fixed-cokernel} gives $N_1$ fixed
points.

For a $3$-cycle $g$, the matrix $M_g-I$ has determinant $3$, and
\[
 \Lambda/(g-I)\Lambda\cong\Z/3\Z.
\]
The induced difference map is $I-\bar\gamma$ on
$(\Z/3\Z)^2$. Its cokernel and kernel have the same order.
To identify the kernel, apply the snake lemma to $I-\gamma$ on
the exact sequence $0\to\Z^2\xrightarrow{3}\Z^2\to\F_3^2\to0$.
Since $I-\gamma$ is injective on $\Z^2$, this gives
\[
 A_\gamma[3]
 \cong\ker(I-\bar\gamma:\F_3^2\to\F_3^2).
\]
Averaging over the identity, three transpositions and two $3$-cycles gives
\eqref{eq:identity-sector}.

\smallskip
\noindent\emph{Transposition sector.}
The centralizer of $\tau$ is $\{e,\tau\}$.  The identity fixes all
$|C_\tau(\gamma)|=N_2$ elements.  Since
$\Lambda/(\tau-I)\Lambda\cong\Z$ and $M_\tau$ becomes the identity on
this quotient, the map $I-M_\tau\otimes\gamma$ descends to
$I-\gamma$.  Thus $\tau$ has $N_1$ fixed points, proving
\eqref{eq:transposition-sector}.

\smallskip
\noindent\emph{$3$-cycle sector.}
In $C(\sigma)=\{e,\sigma,\sigma^2\}$, the identity fixes
$|C_\sigma(\gamma)|=N_3/N_1$ elements.  For
$g=\sigma$ or $\sigma^2$, the quotient
$\Lambda/(g-I)\Lambda$ is $\Z/3\Z$, on which $M_\sigma$ acts trivially.  In both cases the descended difference map
is $I-\bar\gamma$ on $\F_3^2$, so the fixed-point count is
$T_3(\gamma)$.  This proves \eqref{eq:threecycle-sector}.
\end{proof}

\begin{theorem}[The rank-three dimension]
\label{thm:main-formula}
For hyperbolic $\gamma\in\mathrm{SL}_2(\Z)$,
\[
 S_3(\gamma)=\frac{N_1^2}{6}+N_1+\frac{N_2}{2}
   +\frac{N_3}{3N_1}+T_3(\gamma),
 \qquad T_3(\gamma)=V_3.
\]
\end{theorem}

\begin{proof}
Adding the three orbit counts in Proposition~\ref{prop:three-sectors}
gives the formula for $S_3(\gamma)$. The identity-sector calculation
gives $T_3(\gamma)=|\ker(I-\bar\gamma)|=V_3$ by
\eqref{eq:general-observer-identity}.
\end{proof}

\begin{corollary}[Trace form]\label{cor:trace-form}
Let $t=\tr(\gamma)$, with $|t|>2$.  Then
\[
 S_3(\gamma)=
 \begin{cases}
  t^2+t-3+T_3(\gamma),&t>2,\\[2mm]
  t^2-t+1+T_3(\gamma),&t<-2.
 \end{cases}
\]
\end{corollary}

\begin{proof}
The trace recurrence for powers of an $\mathrm{SL}_2$ matrix gives
\[
 N_1=|2-t|,\qquad N_2=t^2-4,\qquad
 \frac{N_3}{N_1}=(t+1)^2.
\]
Substitute in Theorem~\ref{thm:main-formula} for each sign of $t$.
\end{proof}

\subsection{The mod-\texorpdfstring{$3$}{3} correction}

The subgroup $3\Z^2$ defines a degree-nine cover of the torus to which
every toral automorphism lifts. The inclusion of the fiber identifies
the group responsible for the correction with the $3$-torsion in the
first homology of the bundle.

\begin{proposition}[Mod-$3$ observer identity]
\label{prop:mod3-observer}
Let $\gamma\in\mathrm{SL}_2(\Z)$ with $\det(I-\gamma)\neq0$, and
write $\bar\gamma$ for its reduction modulo $3$. Then
\[
 V_{3\Z^2}(f_\gamma)=|\ker(I-\bar\gamma)|=T_3(\gamma)
   =\bigl|H_1(M_\gamma;\Z)[3]\bigr|.
\]
\end{proposition}

\begin{proof}
Since $\det(I-\gamma)\ne0$, all fixed-point classes have the same
nonzero index. Reduction maps $A_\gamma$ onto $\coker(I-\bar\gamma)$,
so every class in that cokernel is visible. Its order equals
$|\ker(I-\bar\gamma)|$. The snake lemma identifies this kernel with
$A_\gamma[3]$, as in the identity-sector calculation above.
The Wang sequence for the bundle gives
\[
 0\longrightarrow A_\gamma\longrightarrow H_1(M_\gamma;\Z)
 \longrightarrow\Z\longrightarrow0,
\]
where the first map is induced by inclusion of the fiber. Since
$A_\gamma$ is finite, its image is the torsion subgroup of $H_1(M_\gamma;\Z)$.
\end{proof}

\begin{proposition}\label{prop:t3-classification}
Let $\gamma\in\mathrm{SL}_2(\Z)$ with $\det(I-\gamma)\neq0$, and let
$\bar\gamma$ be its reduction modulo $3$.  Then
\[
 T_3(\gamma)
 =
 |\ker(I-\bar\gamma:\F_3^2\to\F_3^2)|
 =
 \begin{cases}
 1,&\tr(\gamma)\not\equiv2\pmod3,\\[2mm]
 3,&\tr(\gamma)\equiv2\pmod3
       \ \text{and}\ \bar\gamma\neq I,\\[2mm]
 9,&\bar\gamma=I.
 \end{cases}
\]
\end{proposition}

\begin{proof}
The first equality follows from Proposition~\ref{prop:mod3-observer}.
Since $\det\bar\gamma=1$,
\[
 \det(I-\bar\gamma)=1-\tr(\bar\gamma)+\det\bar\gamma
 =2-\tr(\bar\gamma).
\]
If this determinant is nonzero, the kernel has order $1$. Otherwise
$I-\bar\gamma$ has rank zero or one according as $\bar\gamma=I$ or
$\bar\gamma\ne I$. The kernel then has order $9$ or $3$, respectively.
\end{proof}

\begin{theorem}[Classification at fixed Nielsen data]
\label{thm:classification}
Let $\gamma,\gamma'\in\mathrm{SL}_2(\Z)$ be hyperbolic with the same
trace.  Then
\[
 N(f_\gamma^k)=N(f_{\gamma'}^k)
 \qquad(k\geq1),
\]
where $f_\gamma,f_{\gamma'}:T^2\to T^2$ are the induced toral
automorphisms. Their rank-three summands satisfy
\[
 S_3(\gamma)-S_3(\gamma')=T_3(\gamma)-T_3(\gamma').
\]
The residue of the common trace modulo $3$ gives the following cases.
\begin{enumerate}
\item If the common trace is not congruent to $2$ modulo $3$, the two
      dimensions agree.
\item If the common trace is congruent to $2$ modulo $3$, the dimensions
      differ by $6$ in absolute value precisely when exactly one of
      $\gamma,\gamma'$ reduces to the identity modulo $3$. Otherwise
      they agree.
\end{enumerate}
\end{theorem}

\begin{proof}
For matrices in $\mathrm{SL}_2(\Z)$, the sequence
$a_k=\tr(\gamma^k)$ is determined by
\[
 a_0=2,\qquad a_1=\tr(\gamma),\qquad
 a_{k+1}=\tr(\gamma)a_k-a_{k-1}.
\]
Thus equal traces give equal traces of all powers.  For a hyperbolic
toral automorphism, every fixed-point class of every iterate is
essential and
\[
 N(f_\gamma^k)=|\det(I-\gamma^k)|
              =|2-\tr(\gamma^k)|.
\]
This gives the same periodic Nielsen sequence for both automorphisms
\cite{Jiang}. Subtract the formulas in Theorem~\ref{thm:main-formula}
and apply Proposition~\ref{prop:t3-classification}.
\end{proof}

\begin{remark}[Periodic Nielsen data and trace]\label{rem:nielsen-trace}
Put $t=\tr(\gamma)$. For hyperbolic $\gamma$, the trace recurrence determines every
$N(f_\gamma^k)=|2-\tr(\gamma^k)|$. Conversely,
\[
 N(f_\gamma)=|2-t|,\qquad N(f_\gamma^2)=t^2-4
\]
determine $|t|$ and its sign.
\end{remark}

\subsection{An infinite family of torus bundles}

We construct pairs of torus bundles with the same periodic Nielsen
numbers and different rank-three summand dimensions. To compare the
covers at each iterate, recall the visibility profile of a map $f:X\to X$
\cite[Definition~5.1]{Kaya}
\[
 V_f(B)=
 \max\{V_H(f):H\ \text{is a finite observer with}\ [\pi_1(X):H]\leq B\}
\]
for a budget $B\geq1$. On the torus the maximum runs over the
$f_*$-invariant subgroups of $\Z^2$ of index at most $B$.
For a positive integer $D$, put
\[
 \delta_D(B)=\max\{d:d\mid D,\ d\leq B\}.
\]
For a hyperbolic toral automorphism $\eta$, the visibility profile is determined
by the divisors of $|\det(I-\eta)|$ \cite[Theorem~8.2]{Kaya}.

\begin{lemma}[Exact toral visibility profile]\label{lem:toral-profile}
Let $\eta\in\mathrm{SL}_2(\Z)$ be hyperbolic and put
$D=\lvert\det(I-\eta)\rvert$.  Then for every $B\geq1$,
\[
 V_{f_\eta}(B)=\delta_D(B).
\]
\end{lemma}

\begin{proof}
Write $C=\coker(I-\eta)$, so $|C|=D$.  Twisted conjugacy in the
abelian group $\Z^2$ reduces to cosets of $\im(I-\eta)$, so the
Reidemeister classes of $f_\eta$ are canonically indexed by $C$.
Every class is essential with the same local index
$\operatorname{sgn}\det(I-\eta)$, so no cancellation occurs when the
Reidemeister trace is pushed to any finite observer, and for an
$\eta$-invariant subgroup $H\leq\Z^2$ of finite index, $V_H(f_\eta)$
is the order of the image of the natural homomorphism
\[
 C\longrightarrow
 \coker\bigl(I-\bar\eta:\Z^2/H\to\Z^2/H\bigr).
\]
This image is a quotient of $C$, so its order divides $D$; it also has
order at most $\lvert\Z^2/H\rvert=[\Z^2:H]\leq B$.  Hence
$V_{f_\eta}(B)\leq\delta_D(B)$.

Conversely, let $d\mid D$.  Choose a surjection
$\pi_d:C\twoheadrightarrow Q_d$ onto an abelian group of order $d$,
which exists because $C$ is finite abelian, and let $H_d\leq\Z^2$ be
the preimage of $\ker\pi_d$ under the projection
$\Z^2\twoheadrightarrow C$.  Since $\eta$ acts as the identity on $C$
(indeed $\eta x-x\in\im(I-\eta)$ for every $x$), the subgroup $H_d$ is
$\eta$-invariant of index $d$, the induced endomorphism on
$\Z^2/H_d\cong Q_d$ is the identity, and the class map
$C\to\coker(I-\mathrm{id})=Q_d$ is the surjection $\pi_d$.  Hence
$V_{H_d}(f_\eta)=d$.  Maximizing over divisors $d\leq B$ gives the
reverse inequality.
\end{proof}

\begin{theorem}\label{thm:infinite-family}
For every $r\in\Z\setminus\{0\}$, set
\[
 \gamma_{1,r}=
 \begin{pmatrix}
 0&-1\\
 1&9r+2
 \end{pmatrix},
 \qquad
 \gamma_{2,r}=
 \begin{pmatrix}
 1&3\\
 3r&9r+1
 \end{pmatrix}.
\]
These matrices have the following properties.
\begin{enumerate}
\item $\gamma_{1,r}$ and $\gamma_{2,r}$ are hyperbolic elements of
      $\mathrm{SL}_2(\Z)$ with common trace $9r+2$;
\item for every $N\geq1$,
      \[
      \dim\Sk_{\mathrm{GL}_N}(M_{\gamma_{1,r}})
      =
      \dim\Sk_{\mathrm{GL}_N}(M_{\gamma_{2,r}});
      \]
\item their $\mathrm{SL}_3$ quantum-torus summands satisfy
      \[
      \dim\HH_0^{\gamma_{2,r}}
        \bigl(\SkAlg_{\mathrm{SL}_3}(T^2)\bigr)
      -
      \dim\HH_0^{\gamma_{1,r}}
        \bigl(\SkAlg_{\mathrm{SL}_3}(T^2)\bigr)
      =6.
      \]
\item their periodic Nielsen numbers and Nielsen zeta functions are
      identical; moreover, for every $k,B\geq1$,
      \[
       V_{f_{\gamma_{1,r}}^k}(B)
       =
       V_{f_{\gamma_{2,r}}^k}(B),
      \]
      so their finite-cover visibility profiles agree at every iterate.
\end{enumerate}
The associated torus bundles are non-homeomorphic.
\end{theorem}

\begin{proof}
Both matrices have determinant one and trace $9r+2$, whose absolute
value exceeds $2$ for $r\ne0$. Proposition~\ref{prop:gln-freeness}
therefore gives (ii), also a consequence of the partition function in
\cite{BierentJordanVancraeynestVazirani}.

Modulo $3$,
\[
 \bar\gamma_{1,r}=
 \begin{pmatrix}0&-1\\1&2\end{pmatrix}\neq I,
 \qquad
 \bar\gamma_{2,r}=I.
\]
Proposition~\ref{prop:t3-classification} therefore gives
\[
 V_{3\Z^2}(f_{\gamma_{1,r}})
 =T_3(\gamma_{1,r})=3,
 \qquad
 V_{3\Z^2}(f_{\gamma_{2,r}})
 =T_3(\gamma_{2,r})=9.
\]
Their difference gives (iii) by Theorem~\ref{thm:classification}.

The equality of all periodic Nielsen numbers follows from
Theorem~\ref{thm:classification}, and therefore their Nielsen zeta
functions
\[
 \zeta_N(f;z)=
 \exp\left(\sum_{k\geq1}\frac{N(f^k)}{k}z^k\right)
\]
coincide. Applying Lemma~\ref{lem:toral-profile} to each hyperbolic
iterate $\eta^k$ gives
\[
 V_{f_\eta^k}(B)=\delta_{|\det(I-\eta^k)|}(B).
\]
The equal Nielsen numbers thus give equal profiles for all $k$ and $B$.

Finally, the entries of the two difference matrices have greatest
common divisors $1$ and $3$, respectively, and each determinant has
absolute value $9|r|$. Their Smith normal forms give
\[
 \coker(I-\gamma_{1,r})\cong\Z/(9|r|)\Z,
 \qquad
 \coker(I-\gamma_{2,r})
 \cong\Z/3\Z\oplus\Z/(3|r|)\Z.
\]
Only the first group is cyclic. These groups are the torsion subgroups
of $H_1$ of the two mapping tori, so the manifolds cannot be homeomorphic.
\end{proof}

\begin{remark}[Mutant torus bundles]
\label{rem:bjvv-comparison}
In the shear example of \cite[Remark~1.9]{BierentJordanVancraeynestVazirani},
the mapping tori of a shear and its negative have different
$\mathrm{GL}_2$-skein dimensions.
The pairs in Theorem~\ref{thm:infinite-family} have equal
$\mathrm{GL}_N$ dimensions for all $N$. The distinction occurs in
the $\mathrm{SL}_3$ empty-skein summand.
\end{remark}

\begin{example}
For $r=-1$, the matrices are
\[
 \gamma_{1,-1}=
 \begin{pmatrix}0&-1\\1&-7\end{pmatrix},
 \qquad
 \gamma_{2,-1}=
 \begin{pmatrix}1&3\\-3&-8\end{pmatrix}.
\]
Their common first three Nielsen numbers are
\[
 N_1=9,\qquad N_2=45,\qquad N_3=324.
\]
Theorem~\ref{thm:main-formula} becomes
\[
 \dim\HH_0^\gamma\bigl(\SkAlg_{\mathrm{SL}_3}(T^2)\bigr)
 =57+V_{3\Z^2}(f_\gamma).
\]
Both cokernels have order $9$. The cyclic cokernel has three elements
killed by $3$, and all nine elements of the other cokernel are killed by
$3$. The empty-skein dimensions are therefore $60$ and $66$. Thus the
degree-nine congruence cover detects a difference even though the
optimized visibility profiles agree at every iterate.

The negatives of these same matrices illustrate
Corollary~\ref{cor:cyclic-reconstruction}. Both have trace $7$, and
each $I+\gamma_{i,-1}$ has cyclic cokernel of order $5$. Their
empty-skein dimensions therefore agree in every rank. On the double covers the common
trace is $47$. The contents of $\gamma_{1,-1}^2-I$ and
$\gamma_{2,-1}^2-I$ are $1$ and $3$, respectively, so the torsion
groups are $\Z/45\Z$ and $\Z/3\Z\oplus\Z/15\Z$.
Corollary~\ref{cor:trace-form} gives rank-three dimensions $2256$
and $2262$ on these covers.
\end{example}

\end{document}